\documentclass[12pt,a4paper]{article}

\usepackage[left=1in, right=1in, top=1in, bottom=1in]{geometry}

\usepackage{amsthm,amsmath,natbib,amssymb,amsfonts,bm}
\usepackage{graphicx}
\usepackage{placeins}
\usepackage{subcaption}

\newcommand{\bX}{\boldsymbol{X}}

\newcommand{\blambda}{\boldsymbol{\lambda}}

\usepackage{caption}
\usepackage{setspace}
\usepackage{etoolbox}
\usepackage[labelfont=bf]{caption}
\usepackage[labelsep=period]{caption}
\AtBeginEnvironment{tabular}{\doublespacing}

\usepackage{titlesec}
\titleformat{\section}
{\normalfont\large\bfseries}{\thesection.}{0.5em}{\setstretch{1.5}}
\titleformat{\subsection}
{\normalfont\normalsize\bfseries\itshape}{\thesubsection.}{0.5em}{}

\newtheorem{theorem}{Theorem}

\newtheorem*{theorem*}{Theorem}

\theoremstyle{definition}

\theoremstyle{remark}

\DeclareMathAlphabet{\mathpzc}{OT1}{pzc}{m}{it}

\newcommand{\var}{\text{var}}
\newcommand{\corr}{\text{corr}}

\begin{document}

\baselineskip=24pt

\begin{center}
{\bf \Large Confidence intervals centred on bootstrap smoothed estimators: an impossibility result}
\end{center}

\medskip

\begin{center}
{\bf \large Paul Kabaila$^*$ and Christeen Wijethunga}
\end{center}

\medskip

\begin{center}
{\sl Department of Mathematics and Statistics, La Trobe University, Melbourne, Australia}
\end{center}

\vspace{2cm}


\vspace{12cm}

\noindent * Corresponding author. Department of Mathematics and Statistics,
La Trobe University, Victoria 3086, Australia. Tel.: +61 3 9479 2594; fax +61 3 9479 2466.
{\sl E-mail address:} P.Kabaila@latrobe.edu.au.

\newpage


\noindent \textbf{ABSTRACT}

\medskip
 

\noindent Recently, Kabaila and Wijethunga assessed the performance of a confidence interval centred on a bootstrap smoothed estimator, with width proportional to an estimator
of Efron's delta method approximation to the standard deviation of this estimator. They used a testbed situation consisting of two nested linear regression models, with error variance assumed known, and model selection using a preliminary hypothesis test.
This assessment was in terms of coverage and scaled expected length, where the scaling is with respect to the expected length of the usual confidence interval 
with the same minimum coverage probability. 
They found that this confidence interval has scaled expected length  
that (a) has a maximum value that may be much greater than 1 and (b)  is greater than a number slightly less than 1 when the simpler model is correct. We therefore ask the following question. 
For a confidence interval, centred on the bootstrap smoothed estimator, does there exist a formula for its data-based width such that, in this testbed situation, it has the desired minimum coverage and scaled expected length that
(a) has a maximum value that is not too much larger than 1 and
(b) is substantially less than 1 when the simpler model is correct?
Using a recent decision-theoretic performance bound due to Kabaila and Kong, it is shown that the answer to this question is `no' for a wide range of scenarios.



\bigskip

\noindent \textbf{KEYWORDS} \newline
confidence interval; coverage probability; decision theory; bootstrap smoothed estimator;
model selection; nested linear regression models; scaled expected length.


\newpage

\section{\textbf{Introduction}}

A bootstrap smoothed (or bagged, \citeauthor{Breiman1996}, \citeyear{Breiman1996}) estimator is a smoothed version of an estimator found after preliminary data-based model selection. A key result of \cite{Efron2014} is a very convenient and widely applicable formula for a delta method approximation to the standard deviation of this estimator. 
\cite{Efron2014} also considered 
 a confidence interval, with nominal coverage $1-\alpha$, centred on the bootstrap estimator and with half-width equal to the $1 - \alpha/2$ quantile of the standard normal distribution multiplied by an estimate of this approximation to the standard deviation. 
To rigorously assess the performance of this confidence interval, 
\cite{KabailaWijethunga2019a}
 use the following testbed situation. They consider two nested normal linear regression models with known error variance and parameter of interest 
$\theta$, a specified linear combination of the regression parameters.
These two nested models are the full model and the simpler model. The simpler model is obtained when $\tau$,
a distinct specified linear combination of the regression parameters, is set to 0.
The bootstrap smoothed estimator that they consider is a smoothed version of the post-model-selection estimator obtained after a preliminary test of the null 
hypothesis that $\tau = 0$ against the alternative hypothesis that 
$\tau \ne 0$.

For the testbed situation they consider, \cite{KabailaWijethunga2019a} derive computationally convenient exact formulas for the ideal (i.e. in the limit as the number of bootstrap replications diverges to infinity)
bootstrap smoothed estimator and Efron's delta method approximation to the standard deviation of this estimator. 
They also assess the performance of the confidence interval, with nominal coverage $1-\alpha$, centred on the ideal bootstrap estimator and with half-width equal to the $1 - \alpha/2$ quantile of the standard normal distribution multiplied by an estimate of the ideal delta method approximation to the standard deviation. We call this the 
${\bf sd}_{\rm delta}$ {\bf interval}.
This confidence interval has the attractive features that (A1) it has endpoints that are continuous functions of the data and (A2) to an excellent approximation it reverts to the usual $1 - \alpha$ confidence interval based on the full model when the data and the simpler model are highly discordant.

For the testbed situation, \cite{KabailaWijethunga2019a} assess the performance of the 
${\bf sd}_{\rm delta}$ {\bf interval} using its coverage probability and scaled expected length, defined as follows. The scaled expected length of a confidence interval is defined to be its expected length divided by the
expected length of the usual confidence interval, with the same minimum coverage probability, based on the full model. Let 
$\rho$ denote the known correlation between the least squares estimators of $\theta$ and $\tau$.
For $\rho = 0$, the 
${\bf sd}_{\rm delta}$ {\bf interval}
is identical to the usual $1 - \alpha$ confidence interval based on the full model. However, as $|\rho|$ increases, these two confidence intervals increasingly differ from each other.

Define the parameter $\gamma$ to be
$\tau$ divided by the standard deviation of the least squares estimator of $\tau$, so that $\gamma$ is unknown.
For given nominal coverage $1 - \alpha$ and size of preliminary test, both the coverage probability and the scaled expected length of 
the 
${\bf sd}_{\rm delta}$ {\bf interval} are functions of $|\rho|$ and $|\gamma|$. 
Figure 5 of \cite{KabailaWijethunga2019a} and Figures 6--10 of the Supplementary Material for this paper show the following.
The ${\bf sd}_{\rm delta}$ {\bf interval} 
 has scaled expected length that
 (a) has a maximum value that is an increasing function of $|\rho|$ that can be much larger than 1 for large $|\rho|$ and
(b) is greater than a number slightly less than 1 when the simpler model is correct (i.e. when $\gamma = 0$).

In the context of the testbed situation  used by \cite{KabailaWijethunga2019a}, we have sought a formula for the width of a confidence interval centred on the ideal bootstrap smoothed estimator that leads to this confidence interval having scaled expected length
that is substantially less than 1 when the simpler model is correct. We tried five different formulae, including the width being based on the actual standard deviation and the parametric version of the symmetric nonparametric bootstrap confidence
described by \citeauthor{Hall1992} (\citeyear{Hall1992}, Section 3.6). None of these confidence intervals possessed this
desired property.
This then raises the following question:

\begin{quote}
	In the context of this testbed situation, is there any formula for the width of a confidence interval centred on the ideal bootstrap smoothed estimator that leads to this confidence interval having the attractive features (A1) and (A2), the desired minimum coverage probability $1 - \alpha$ and 
	scaled expected length that (a) has a maximum value that is not too much larger than 1 and (b) is substantially less than 1 when the simpler model is correct?
\end{quote}

\noindent We use a performance bound derived by \cite{KabailaKong2016} to answer this question. This bound builds on the earlier work of \cite{Blyth1951},
\cite{HodgesLehmann1952}, \citeauthor{Kempthorne1983} (\citeyear{Kempthorne1983}, \citeyear{Kempthorne1987}, \citeyear{Kempthorne1988}) and \cite{KabailaTuck2008}.
As shown in Section 3 of the present paper, the application of this performance bound is carried out by computing two unfavourable discrete distributions, each consisting of a finite number of probability masses. 
We have programmed this computation in \texttt{R}, using simple code that is based on
the theoretical results described in Sections 3 and 5 and Appendix A.3. 
For the scenarios described in Section 6 the answer to the question above is `no'.


\section{\textbf{Mathematical specification of the question we will answer}}

We consider the testbed situation consisting of 
two nested linear regression models: the full model
${\cal M}_2$ and the simpler model ${\cal M}_1$.
Suppose that the full model ${\cal M}_2$ is given by
\begin{equation*}
\bm{y} = \bm{X} \bm{\beta} + \bm{\varepsilon}
\end{equation*}
where $\bm{y}$ is a random $n$-vector of responses, $\bX$ is a known $n \times p$ matrix with linearly independent columns ($p < n$), $\bm{\beta}$ is an unknown $p$-vector of parameters and $\bm{\varepsilon} \sim N(\bm{0}, \sigma^2 \bm{I}_n)$, with $\sigma^2$ assumed known.
Suppose that $\bm{\beta} = [\theta, \tau, \bm{\lambda}^\top ]^\top$, where $\theta$ is the scalar parameter of interest, $\tau$ is a scalar parameter used in specifying the model ${\cal M}_1$ and $\blambda$ is a ($p-2$)-dimensional parameter vector. 
The model ${\cal M}_1$ is ${\cal M}_2$ with $\tau=0$. As shown by
\cite{KabailaWijethunga2019a}, this scenario can be obtained by a change of parametrisation from a more general scenario.

Let $\widehat{\theta}$ and $\widehat{\tau}$ denote the least squares estimators of $\theta$ and $\tau$, respectively.
Let $v_\theta = \var(\widehat{\theta})/\sigma^2$, $v_\tau = \var(\widehat{\tau})/\sigma^2$ and
$\rho=\corr(\widehat{\theta}, \widehat\tau)$. 
Note that $v_\theta$, $v_\tau$ and $\rho$ are known.
Let $\gamma=\tau/\big(\sigma{v_\tau}^{1/2}\big)$, which is an unknown parameter, and also let $\widehat{\gamma}=\widehat{\tau}/\big(\sigma{v_\tau}^{1/2}\big)$.
Suppose that we carry out a preliminary test of the null hypothesis $\tau = 0$ against the alternative hypothesis
$\tau \neq 0$. 
For given value of
the positive number $d$, suppose that we accept this null hypothesis when $|\widehat{\gamma}| \leq d $; otherwise we reject this null hypothesis.
Let $\widetilde{\alpha}$ denote the size of this preliminary test.
If $|\widehat{\gamma}| \leq d $ we choose model ${\cal M}_1$; otherwise we choose model ${\cal M}_2$.

%

Let
$k(\gamma) 
= \phi(d + \gamma) 
- \phi(d - \gamma) 
+ \gamma \big (\Phi(d- \gamma) - \Phi(-d-\gamma ) \big)$,
where $\phi$ and $\Phi$ denote the $N(0,1)$ pdf and cdf.
\cite{KabailaWijethunga2019a}
prove that 
the ideal (i.e. in the limit as the number of bootstrap replications diverges to infinity) bootstrap smoothed estimator of $\theta$ is
\begin{equation*}
\widetilde{\theta} = \widehat{\theta} - \rho \, \sigma \, v_\theta^{1/2} \, k(\widehat{\gamma}).
\end{equation*}
\cite{Efron2014} derived a delta method approximation to the standard deviation of the bootstrap smoothed estimator.
\cite{KabailaWijethunga2019a} show that, in the limit as the number of bootstrap replications diverges to infinity, this approximation is given by 
$\sigma \, v_{\theta}^{1/2} \, r_{\text{delta}}(\gamma)$, where 
$r_{\rm delta}(\gamma)
= \big( 1 - 2 \rho^2 q(\gamma) + \rho^2 q^2(\gamma) \big)^{1/2}$,
with $q(\gamma) = \Phi(d-\gamma) - \Phi(-d-\gamma ) - d \, \big( \phi(d+\gamma) + \, \phi(d-\gamma ) \big)$.
Let $z(a) = \Phi^{-1}(1 - a/2)$.
The 
confidence interval for $\theta$, with nominal coverage $1-\alpha$, centred on  $\widetilde{\theta}$ and with half-width equal to $z(\alpha)$ multiplied by an estimate of this approximation to the standard deviation is
\begin{equation*}
\left[\widehat{\theta} - \sigma \, v_\theta^{1/2} \, \rho \,  k(\widehat{\gamma})
\pm  \sigma \, v_\theta^{1/2} \, \big(z(\alpha) \,r_{\rm delta}(\gamma) \big)\right].
\end{equation*}
The function $r_{\rm delta}: \mathbb{R} \rightarrow [0, \infty)$ is a continuous even function and $r_{\rm delta}(x) \rightarrow 1$ as $x \rightarrow \infty$.
As shown in the Supplementary Material, 
to an excellent approximation, $r_{\rm delta}(x) = 1$ for all $x \ge c$,
where $c = 10$.

We consider confidence intervals of the form
\begin{equation*}
\text{CI}(s) = \left[\widehat{\theta} - \rho \, \sigma \, v_\theta^{1/2} \, k(\widehat{\gamma})
\pm \sigma \, v_\theta^{1/2} \, s(\widehat{\gamma})\right],	
\end{equation*}
where $s: \mathbb{R} \rightarrow [0, \infty)$ is an even function.
To apply the performance bound of \cite{KabailaKong2016}, we suppose that
$s(x) = z(\alpha)$ for all $x \ge c$.
Let ${\cal D}$ denote the class of functions $s$ 
that satisfy this property. We do not require that the function $s$ is continuous.
In other words, a confidence interval $\text{CI}(s)$, where $s \in {\cal D}$, does not necessarily 
possess the attractive feature (A1). In addition, such a confidence interval does not necessarily possess the attractive feature (A2). This is because 
$\widehat{\gamma} \sim N(0,1)$ under the simpler model, so that the data and this model can be said to be highly discordant when $|\widehat{\gamma}| > 4$, say.

We now introduce the following question:
\begin{quote}
	Is there a confidence interval $\text{CI}(s)$, specified by a function $s$ in 
	${\cal D}$, such that this confidence interval has 
	minimum coverage probability $1 - \alpha$ and scaled expected length 
	that (a) has a maximum value that is not too much larger than 1 and (b) is substantially less than 1 when the simpler model is correct? 	
\end{quote}
If the answer to this question is `no' then the answer to the question posed in the introduction must also be `no'.

\section{\textbf{Exact formulas for the coverage probability and scaled expected length of $\boldsymbol{\text{CI}(s)}$}}

For notational convenience, let $b(x) = \rho\, k(x)$ for all $x \in \mathbb{R}$, so that
\begin{equation*}
\text{CI}(s) = \left[ \widehat{\theta} - \sigma \, v_\theta^{1/2}\, b(\widehat{\gamma}) \pm \sigma \, v_\theta^{1/2} \, s(\widehat{\gamma}) \right].
\end{equation*}
Every $s \in {\cal D}$ is an even function that satisfies $s(x) = z(\alpha)$ for all $x \ge c$. Hence every $s \in {\cal D}$ is 
specified by $s$ restricted to the domain $[0, c]$. In this section we present exact formulas for the coverage probability and   
scaled expected length of $\text{CI}(s)$ in terms of the function $s$ restricted to the domain 
$[0, c]$. 
The function $k$ is a continuous odd function
and $k(x) \rightarrow 0$ as $x \rightarrow \infty$.
As shown in the Supplementary Material, 
to an excellent approximation, $k(x) = 0$ for all $x \ge c$,
where $c = 10$. For computational convenience, we approximate $k(x)$, and therefore $b(x)$, by 0 for all $x \ge c$.

The following theorem, proved in Appendix A.1, provides an exact formula for the coverage probability of the confidence interval $\text{CI}(s)$ in terms of the function $s$ restricted to the domain $[0, c]$. 

\begin{theorem}
	For any given $\rho$ and function $s \in {\cal D}$, the coverage probability of 
	${\rm CI}(s)$ is a function of 
	$\gamma$. We denote this coverage probability by $c(\gamma; s, \rho)$.
	For any given $\rho$ and $s \in {\cal D}$, $c(\gamma; s, \rho)$ is an even function of $\gamma$.
	Also, for any given $\gamma$ and $s \in {\cal D}$, $c(\gamma; s, \rho)$ is an even function of $\rho$.
	Let
 $\ell(h, \gamma; x) = P \big(  b(h) - x \leq \widetilde{G} \leq  b(h) + x \big)$
 and 
 $\ell^\dag(h, \gamma) = P \big( -z(\alpha) \leq \widetilde{G} \leq z(\alpha) \big)$, for 
 $\widetilde{G} \sim N \big(\rho(h-\gamma), 1-\rho^2 \big)$.
 Now let $R_1(s, \gamma)$ denote
 \begin{equation*}
 \int_{0}^{c} \Big( \left( \ell^{\dag}(h, \gamma) - \ell(h, \gamma; s(h)) \right) \, \phi(h-\gamma) +  \left( \ell^{\dag}(-h, \gamma) - \ell(-h, \gamma; s(h)) \right) \, \phi(h+\gamma)\, \Big) dh. 
 \end{equation*}
 Then $c(\gamma; s, \rho) = 1 - \alpha - R_1(s, \gamma)$.
\end{theorem}
\noindent It follows from this theorem that, for any given $s \in {\cal D}$,  $R_1(s, \gamma)$ is an even function of $\gamma$. It also follows from this theorem that, for any given $\gamma$ and $s \in {\cal D}$,  $R_1(s, \gamma)$ is an even function of $\rho$.

Let $I$ denote the usual $1-\alpha$ confidence interval for $\theta$ based on the full model. In other words, 
\begin{equation*}
I 
= \left[ \widehat{\theta}  \pm z(\alpha) \, \sigma \, v_\theta^{1/2}  \right].
\end{equation*}
Define the scaled expected length of $\text{CI}(s)$ to be
\begin{equation*}
\frac{E(\text{length of }\text{CI}(s))}{E(\text{length of } I)}
= \frac{E(s(\widehat{\gamma}))}{z(\alpha)}.
\end{equation*}
It follows from \eqref{bivariate_theta_gamma} that this is a function of $\gamma$ for given function 
$s$. We denote this function by $e(\gamma; s)$. 
The following theorem, proved in Appendix A.2, provides an exact formula for $e(\gamma; s)$ in terms of the function $s$ restricted to the domain $[0, c]$. 

\begin{theorem}
	For given function $s$, the scaled expected length of ${\rm CI}(s)$ is a function of $\gamma$.
	We denote this scaled expected length by $e(\gamma; s)$. Then 
	$e(\gamma; s)  = 1 + R(s, \gamma)$, where
	\begin{equation}
	\label{ExactFormulaRsGamma}
	R(s, \gamma) = \int_{0}^{c} \left( \frac{s(h)}{z(\alpha)} - 1 \right) \big( \phi(h-\gamma) + \phi(h+\gamma) \big) \, dh.
	\end{equation}
\end{theorem}

\noindent Obviously $e(\gamma; s)$ and $R(s, \gamma)$ are even functions of $\gamma$.

The function $R_1(s, \gamma)$ is the probability of non-coverage of $\theta$ by
${\rm CI}(s)$, for true parameter value $\gamma$. The function $R(s, \gamma)$ is the 
scaled expected length of ${\rm CI}(s)$ minus 1, for true parameter value $\gamma$.
Both $R_1(s, \gamma)$ and $R(s, \gamma)$ are risk functions in the following sense.
If both
$R_1(s_1, \gamma) < R_1(s_2, \gamma)$ and $R(s_1, \gamma) < R(s_2, \gamma)$
then ${\rm CI}(s_1)$ is preferred to ${\rm CI}(s_2)$, for true parameter value $\gamma$
(cf. \citeauthor{Kempthorne1987}, \citeyear{Kempthorne1987}, p.172).

\section{Application of the performance bound of \cite{KabailaKong2016}}

A direct answer to the question raised at the end of Section 2 is provided by
finding 
$\inf_{s \in {\cal D}} \;  e(0;s)$,
subject to the following two constraints:

\smallskip

\noindent (C1) Coverage constraint \newline
\indent The coverage constraint is
$c(\gamma; s, \rho) \ge 1 - \alpha$ for all $\gamma \ge 0$. 
We have used here the fact \newline 
\indent that $c(\gamma; s, \rho)$ is an even function of $\gamma$.

\smallskip

\noindent (C2) Maximum Scaled Expected Length constraint \newline
\indent For given $u > 0$, $e(\gamma; s) \le 1 + u$ for all $\gamma \ge 0$. 
We have used here the fact that $e(\gamma; s)$ \newline 
\indent  is an even function of $\gamma$.

\smallskip

\noindent We answer the question raised at the end of Section 2 indirectly by finding a lower bound
on $\inf_{s \in {\cal D}} \;  e(0;s)$, subject to these two constraints. If this lower bound is greater than 1 then the answer to this question is no.

To apply the performance bound of \cite{KabailaKong2016}, we need to express $e(0; s)$ and the constraints (C1) and (C2) in terms of 
risks and integrated risks. 
Since $e(\gamma; s)  = 1 + R(s, \gamma)$, 
\begin{equation}
\label{SEL_Rzero}
\quad \text{Subject to the constraints (C1)and (C2),} \ \
\displaystyle{\inf_{s \in {\cal D}}} \;  e(0;s)
= 1 + \displaystyle{\inf_{s \in {\cal D}}} \; R(s, 0).
\end{equation}
Note that
\begin{equation*}
R(s, 0) 
= 2\int_{0}^{c} \left( \frac{s(h)}{z(\alpha)} - 1 \right)\, \phi(h)\, dh
= \int_{-\infty}^{\infty} R(s, \gamma) \, d \, \pi(\gamma),
\end{equation*}
where $\pi$ is the cumulative distribution function of the distribution with a unit probability point mass at 0, i.e.
$\pi(x) = 1$ for $x \ge 0$ and $\pi(x) = 0$ for $x < 0$.
%
%
Thus $R(s, 0)$ is an integrated risk. 
To bring the notation in line with that used by \cite{KabailaKong2016},
let $R_2(s, \gamma) \equiv R(s, \gamma)$.
We now express the two constraints as follows.

\smallskip

\noindent (C1) Coverage constraint \newline
\indent The coverage constraint is
$R_1(s, \gamma) \leq 0$ for all $\gamma \ge 0$.

\smallskip

\noindent (C2) Maximum Scaled Expected Length constraint \newline
\indent For given $u > 0$, $R_2(s, \gamma)\leq u$ for all $\gamma \ge 0$.

\smallskip

We now apply Theorem 2(a) and the method described in Appendix C of \cite{KabailaKong2016}. 
Let $m_1$ and $m_2$ be given positive integers. Suppose that $\gamma_1(1), \dots, \gamma_1(m_1)$ and $\gamma_2(1), \dots, \gamma_2(m_2)$ satisfy
\begin{align}
\label{comp_ineq1}
\begin{split}
&0 \le \gamma_1(1) < \gamma_1(2) < \dots < \gamma_1(m_1)
\\
&0 < \gamma_2(1) < \gamma_2(2) < \dots < \gamma_2(m_2).
\end{split}
\end{align} 
Introduce the nonnegative variables $\nu_1(1), \dots, \nu_1(m_1)$ and $\nu_2(1), \dots, \nu_2(m_2)$ and let
\begin{align*}
\bm{\gamma} &= \big( \gamma_1(1), \dots, \gamma_1(m_1), \gamma_2(1), \dots, \gamma_2(m_2) \big)
\\
\boldsymbol\nu &= \big( \nu_1(1), \dots, \nu_1(m_1), \nu_2(1), \dots, \nu_2(m_2) \big).
\end{align*}
A given value of $(\bm{\gamma},\boldsymbol\nu)$ corresponds to the following two prior distributions:

\begin{enumerate}
	
	\item[(1)] 
	
	A discrete prior distribution that consists of $m_1$ probability point masses
	$p_1(1), \dots,p_1(m_1)$ located at $\gamma_1(1), \dots ,\gamma_1(m_1)$
	respectively, where
	\begin{equation*}
	p_1(j) = \dfrac{\nu_1(j)}{\sum_{k=1}^{m_1} \nu_1(k)}
	\qquad (j=1, \dots, m_1).
	\end{equation*}

		\item[(2)] 
	
	A discrete prior distribution that consists of $m_2$ probability point masses
	$p_2(1), \dots,p_2(m_2)$ located at $\gamma_2(1), \dots ,\gamma_2(m_2)$
	respectively, where
	\begin{equation*}
	p_2(j) = \dfrac{\nu_2(j)}{\sum_{k=1}^{m_2} \nu_2(k)}
	\qquad (j=1, \dots, m_2).
	\end{equation*}

\end{enumerate}

Let
\begin{align*}
\widetilde{g}(s, \bm{\gamma}, \boldsymbol\nu) = \int_{-\infty}^{\infty} R(s, \gamma)\, d\, \pi (\gamma) 
+ \sum_{j=1}^{m_1} \nu_1(j)\, R_1(s, \gamma_1(j)) 
+ \sum_{j=1}^{m_2} \nu_2(j)\, R_2(s, \gamma_2(j)). 
\end{align*}
Now let $s(\bm{\gamma}, \bm{\nu})$ denote a value of $s \in {\cal D}$ that minimizes $\widetilde{g}(s, \bm{\gamma}, \boldsymbol\nu)$. It follows from 
Theorem 2(a) and the method described in Appendix C of \cite{KabailaKong2016} that
a lower bound for $\inf_{s \in {\cal D}}\; R(s, 0)$,
 subject to the constraints (C1) and (C2),
is 
\begin{equation}
\label{lowerboundfinal}
\widetilde{g}\big(s(\bm{\gamma}, \boldsymbol{\nu}), \bm{\gamma}, \boldsymbol\nu \big) - \sum_{j=1}^{m_2} \nu_2(j) \, u,
\end{equation}  
where $u$ is the given positive number
in the description of the constraint (C2). 
As proposed by \cite{KabailaKong2016}, the lower bound
\eqref{lowerboundfinal} is numerically maximized with respect to $(\bm{\gamma},\boldsymbol\nu)$,
to tighten this lower bound. This amounts to computing 
two unfavorable discrete prior distributions.
We then use \eqref{SEL_Rzero} to obtain the desired lower bound.

The crucial part of this procedure is the computation of $s(\bm{\gamma}, \bm{\nu})$, a value of $s \in {\cal D}$ that minimizes $\widetilde{g}(s, \bm{\gamma}, \boldsymbol\nu)$, for given $(\bm{\gamma},\boldsymbol\nu)$. This computation, which requires some care, is described in the next section.

\section{Computation of $s(\bm{\gamma}, \bm{\nu})$ for given $(\bm{\gamma},\boldsymbol\nu)$}

Throughout this section we suppose that $(\bm{\gamma},\boldsymbol\nu)$ is given. 
We describe the computation of $s(\bm{\gamma}, \bm{\nu})$, a value of $s \in {\cal D}$ that minimizes $\widetilde{g}(s, \bm{\gamma}, \boldsymbol\nu)$.
Straightforward manipulations show that
\begin{equation}
\label{IntegralForGiven_s}
\widetilde{g}\big(s, \bm{\gamma}, \boldsymbol\nu \big) 
= \int_{0}^{c}  q \big(s(h); h, \bm{\gamma}, \boldsymbol\nu \big) \, dh, 
\end{equation}
where $q(x; h, \bm{\gamma}, \boldsymbol{\nu})$ is defined to be 
\begin{align*}
& \left( \frac{x}{z(\alpha)} - 1 \right)\, \Big( 2\phi(h) + \sum_{j=1}^{m_2} \nu_2(j)\, \big( \phi(h-\gamma_2(j)) + \phi(h+\gamma_2(j)) \big) \Big)\\
&\qquad + \sum_{j=1}^{m_1} \nu_1(j) \Big[ \left( \ell^{\dag}(h, \gamma_1(j)) - \ell(h, \gamma_1(j); x )\right) \, \phi(h-\gamma_1(j)) \\ 
&\qquad \qquad \qquad \qquad +  \left( \ell^{\dag}(-h, \gamma_1(j)) - \ell(-h, \gamma_1(j); x ) \right) \, \phi(h+\gamma_1(j))\, \Big].
\end{align*}
It follows from \eqref{IntegralForGiven_s} that a function
 $s(\bm{\gamma}, \boldsymbol{\nu})$, defined as a minimizer of 
$\widetilde{g}(s, \bm{\gamma}, \boldsymbol\nu)$ over $s \in {\cal D}$, may be found as follows. We set $s(\bm{\gamma}, \boldsymbol{\nu})$, evaluated at any $h \in [0, c]$, to be a minimizer over $x\in[0, \infty)$ of $q(x; h, \bm{\gamma}, \boldsymbol{\nu})$.

Now $q(x; h, \bm{\gamma}, \boldsymbol{\nu})$ is a continuous function of $x \in [0, \infty)$ for 
all $h \in [0,c]$ and every given $(\bm{\gamma}, \boldsymbol{\nu})$.
 An examination of some examples of this function of $x \in [0, \infty)$ 
 show that this function may have several local minima, including the possibility of a local minimum at $x = 0$. 
 Consequently, the value of $x \in [0, \infty)$ that minimizes $q(x; h, \bm{\gamma}, \boldsymbol{\nu})$ may change discontinuously, as $h$ increases.
 In other words, the function $s(\bm{\gamma}, \boldsymbol{\nu})$ of $h$
 may have discontinuities. 
 Figure 5 of the Supplementary Material provides some illustrations of functions 
 $q(x; h, \bm{\gamma}, \boldsymbol{\nu})$ of $x\in[0, \infty)$ that have two local minima.
 Figure 4 of the Supplementary Material provides an illustration of a function $s(\bm{\gamma}, \boldsymbol{\nu})$ of $h$ with discontinuities.

 To evaluate the lower bound \eqref{lowerboundfinal}, we need to evaluate 
 \begin{equation}
 \label{IntegralFinal}
 \widetilde{g}\big(s(\bm{\gamma}, \boldsymbol{\nu}), \bm{\gamma}, \boldsymbol\nu \big) 
 = \int_{0}^{c}  q \Big(s(\bm{\gamma}, \boldsymbol{\nu})\text{ evaluated at }h; h, \bm{\gamma}, \boldsymbol\nu \Big) \, dh. 
 \end{equation}
 Although the function $s(\bm{\gamma}, \boldsymbol{\nu})$ of $h$
 may have discontinuities,
the integrand of the integral on the right-hand side of \eqref{IntegralFinal} is a continuous function of $h \in [0, c]$.
An illustration of this, when the function $s(\bm{\gamma}, \boldsymbol{\nu})$ of $h$
has discontinuities, is provided by Figure 3 of the Supplementary Material.  

To carry out the computation of the function $s(\bm{\gamma}, \bm{\nu})$
accurately and effectively, we use the properties of 
$dq(x; h, \bm{\gamma}, \boldsymbol{\nu})/dx$,
considered as a function of $x$, described in Appendix A.3.
Suppose that $h \in [0, c]$ is given. Theorem 3 of Appendix A.3 leads to the procedure described at the end of this appendix for finding an interval $\big[0, \widetilde{x}\big]$
that must contain a value of $x \ge 0$ that minimizes $q(x; h, \bm{\gamma}, \boldsymbol{\nu})$. 

 We use the following two step procedure to find the value of $x \in [0, \widetilde{x}]$ that minimizes $q(x; h, \bm{\gamma}, \boldsymbol{\nu})$. We find all possible local minima in Step 1 and compare them to find the global minimum in Step 2.

 \smallskip
 \noindent {\bf Step 1}: 
 By considering $dq(x; h, \bm{\gamma}, \boldsymbol{\nu})/dx$, find all the local minimizers of $q(x; h, \bm{\gamma}, \boldsymbol{\nu})$ in the interval $[0, \widetilde{x}]$.
 Define $w$ to be the smallest integer that is greater than or equal to 
 $10 \, \widetilde{x}$. We evaluate $dq(x; h, \bm{\gamma}, \boldsymbol{\nu})/dx$
 on the evenly-spaced grid $x_1=0, x_2=0.1, x_3=0.2, \dots, x_w$ of values of $x$.
 To find the values of $x \in [0, x_w]$ that are local minimizers of $q(x; h, \gamma, \nu)$, we need to consider the following two cases.
 
 \noindent \underline{Case 1: $x = 0$}
 
 \noindent $x=0$ is a local minimizer of $q(x; h, \bm{\gamma}, \boldsymbol{\nu})$ if either 
$dq(0; h, \bm{\gamma}, \boldsymbol{\nu})/dx > 0$
 or
 $dq(0; h, \bm{\gamma}, \boldsymbol{\nu})/dx = 0$ and
 $dq(x_2; h, \bm{\gamma}, \boldsymbol{\nu})/dx > 0$;
 otherwise $x=0$ is not a local minimizer.
 
 \noindent \underline{Case 2: $0 < x < x_w$}
 
 \noindent If $dq(x_i; h, \bm{\gamma}, \boldsymbol{\nu})/dx < 0$ and 
 $dq(x_{i+1}; h, \bm{\gamma}, \boldsymbol{\nu})/dx > 0$, then 
 $dq(x; h, \bm{\gamma}, \boldsymbol{\nu})/dx$ has a zero in the interval $[x_i, x_{i+1}]$
 that is a local minimizer of $q(x; h, \bm{\gamma}, \boldsymbol{\nu})$. We find this zero using the \texttt{R} function \texttt{uniroot}, to which we provide the interval $[x_i, x_{i+1}]$.  
 Also, if 
$dq(x_i; h, \bm{\gamma}, \boldsymbol{\nu})/dx = 0$
and $dq(x_{i-1}; h, \bm{\gamma}, \boldsymbol{\nu})/dx < 0$ and 
$dq(x_{i+1}; h, \bm{\gamma}, \boldsymbol{\nu})/dx > 0$ then 
$x_i$ is a zero of 
$dq(x; h, \bm{\gamma}, \boldsymbol{\nu})/dx$ that is a local minimizer of $q(x; h, \bm{\gamma}, \boldsymbol{\nu})$.

 \medskip
 \noindent {\bf Step 2}: Evaluate $q(x; h, \bm{\gamma}, \boldsymbol{\nu})$ at the local minimizers of $q(x; h, \bm{\gamma}, \boldsymbol{\nu})$ found in Step 1. The
 global minimum of $q(x; h, \bm{\gamma}, \boldsymbol{\nu})$ is simply the minimum of the local minima. 
 
\section{\textbf{Numerical results}}

As noted in Section 3, for any given $\gamma$ and $s \in {\cal D}$,  $R_1(s, \gamma)$ is an even function of $\rho$. Since $R(s, \gamma)$ does not depend on $\rho$, it is therefore sufficient to consider $\rho \in [0, 1)$. For any given size
$\widetilde{\alpha}$ of the preliminary test, desired minimum coverage 
$1 - \alpha$ and $\rho \in [0, 1)$, we proceed as follows.

 Let $Q(m_1, m_2)$
denote the set of possible values of $(\bm{\gamma},\boldsymbol\nu)$, for positive integers
$m_1$ and $m_2$. Also let
${\rm LB}(u; m_1, m_2, \bm{\gamma},\boldsymbol\nu)$ denote the lower bound
\eqref{lowerboundfinal}. We will make use of the following easily-proved, but very useful result.

\begin{theorem}
	For any given positive integers
	$m_1$ and $m_2$ and $(\bm{\gamma},\boldsymbol\nu) \in Q(m_1, m_2)$
	such that $\sum_{j=1}^{m_2} \nu_2(j) > 0$ the following result is true.
	The lower bound ${\rm LB}(u; m_1, m_2, \bm{\gamma},\boldsymbol\nu)$
	is a decreasing function of $u \in (0, \infty)$.
\end{theorem}

\begin{proof}
	Suppose that the positive integers
	$m_1$ and $m_2$ and $(\bm{\gamma},\boldsymbol\nu) \in Q(m_1, m_2)$,
	satisfying $\sum_{j=1}^{m_2} \nu_2(j) > 0$, are given. The function 
	$s(\bm{\gamma}, \bm{\nu})$ does not depend on $u > 0$. 
	The result now follows from the expression \eqref{lowerboundfinal}.
\end{proof}

Using the procedure described in Section S3 of the Supplementary Material, we find
`good values' of the positive integers $m_1^*$ and $m_2^*$, $u^* > 0$
and  $(\bm{\gamma}^*,\boldsymbol\nu^*) \in Q(m_1^*, m_2^*)$.
Now define $u^{**}$ to be the solution for $u$ of 
\begin{equation*}
1 + \widetilde{g}\big(s(\bm{\gamma}^*, \boldsymbol{\nu}^*), \bm{\gamma}^*, \boldsymbol\nu^* \big) - \sum_{j=1}^{m_2} \nu_2^*(j) \, u = 1.005.
\end{equation*}  
In other words, 
\begin{equation*}
u^{**}
= \dfrac{\widetilde{g}\big(s(\bm{\gamma}^*, \boldsymbol{\nu}^*), \bm{\gamma}^*, \boldsymbol\nu^* \big) - 0.005}
{\sum_{j=1}^{m_2} \nu_2^*(j)}.
\end{equation*}  
If $u^{**} > 0$ then Theorem 3 implies that the answer to the question stated at the end of Section 2 is `no' for all $u$ satisfying $0 < u \le u^{**}$. Consequently, if 
$u^{**} > 0$ then for all $u$ satisfying 
$0 < u \le u^{**}$ the answer to the question stated in the introduction is `no'.

Section S2 of the Supplementary Material describes the details of computations using the \texttt{R} programming language. In Section S4 of
the Supplementary Material  we provide a method to check the accuracy of the final computed results and describe this procedure using an example. Table \ref{BoundingResults} describes the values of $u^{**}$ computed using the procedure described in 
Section S3 of the Supplementary Material, with $\epsilon = 0.05$. All of the results are for nominal coverage $1-\alpha = 0.95$, $\widetilde{\alpha} \in \{ 0.05, 0.1 \}$ and $|\rho| \in \{ 0.5, 0.6, 0.7, 0.8 \}$. This table also lists the values of $m_1$ and $m_2$, the numbers of probability point masses in the discrete prior distributions (described in Section 4) that relate to the coverage probability and scaled expected length, respectively.

\begin{table}[h!]
	\centering 
	\caption{Values of $u^{**}$ such that 
		${\rm LB}(u; m_1, m_2, \bm{\gamma},\boldsymbol\nu) > 1$ for all $u$ satisfying $0 < u \le u^{**}$. These values were computed using the procedure described in 
		Section S3 of the Supplementary Material, with $\epsilon = 0.05$. We consider $\widetilde{\alpha} \in \{ 0.05, 0.1 \}$, $|\rho| \in \{ 0.5, 0.6, 0.7, 0.8 \}$ and nominal coverage $1-\alpha = 0.95$. Here $m_1$ and $m_2$ are the number of probability point masses in the two prior distributions described in Section 4.} 
	\vspace{2mm}
	\begin{tabular}{ p{2cm} p{2cm} p{2cm}  p{2cm}  p{3cm} } 
		\hline
		$\widetilde{\alpha}$ & $|\rho|$ & $m_1$ & $m_2$ &  $u^{**}$ \\ [0.5ex]
		\hline
		0.05 & 0.5 & 4 & 4 & 0.05268182 \\ 
		& 0.6 & 4 & 4 & 0.07535683 \\
		& 0.7 & 5 & 3 & 0.11375010 \\
		& 0.8 & 5 & 3 & 0.15037320 \\ [0.5ex] 
		\hline
		0.1 & 0.5 & 4 & 2 &  0.02763119 \\ 
		& 0.6 & 7 & 4 & 0.04430236  \\
		& 0.7 & 5 & 2 & 0.06345934 \\
		& 0.8 & 5 & 2 & 0.07750792 \\ [0.5ex] 
		\hline
	\end{tabular}
	\label{BoundingResults}
\end{table}

In view of the question that we have posed in the introduction and at the end of Section
2, it is reasonable to measure the performance of a confidence interval 
CI($s$) that satisfies constraints (C1) and (C2) as follows.
Define the squared scaled expected length gain (\textit{gain}) to be
\begin{equation*}
1 - \big( e(0; s) \big)^2
\end{equation*}
and the squared scaled expected length loss (\textit{loss}) to be
\begin{equation*}
(1 + u)^2 - 1 = u^2 + 2u.
\end{equation*}
Ideally, the \textit{gain} is large and the \textit{loss} is small.
If the positive number $\ell$ is a lower bound on $\inf_{s \in {\cal D}} e(0;s)$, subject to the constraints (C1) and (C2), then $1-\ell^2$ is an upper bound on the \textit{gain} $1 - \big(e(0;s)\big)^2$ for all $s \in {\cal D}$ such that constraints (C1) and (C2) are satisfied.
For the values of $u$ listed in Table \ref{GainLossResults}, we used the values of $m_1$ and $m_2$ in Table 1 that were used to compute $u^{**}$, to carry out Step 1 of the procedure described in 
Section S3 of the Supplementary Material. 
Table \ref{GainLossResults} gives the resulting upper bound on the \textit{gain} for \textit{loss} specified by the given value of $u > 0$. 
As can be seen from this table, 
the ratio (upper bound on \textit{gain})/\textit{loss} is small for all $|\rho| \ge 0.6$.

\begin{table}[h!]
	\centering
	\caption{Computed upper bound on the \textit{gain} for \textit{loss} specified by the given value of $u > 0$. These values were computed using the procedure described in 
		Section S3 of the Supplementary Material, with $\epsilon = 0.05$. We consider  $\widetilde{\alpha} \in \{ 0.05, 0.1 \}$, $|\rho| \in \{ 0.5, 0.6, 0.7, 0.8 \}$ and nominal coverage $1-\alpha = 0.95$. } 
	\vspace{2mm}
	\begin{tabular}{ p{1cm}  p{1cm} p{1.25cm} p{3cm} p{1.75cm} p{3cm}} 
		\hline
		$\widetilde{\alpha}$ & $|\rho|$ & $u$ & upper bound & \textit{loss} & (upper bound on \\
		&  &  & on \textit{gain} & & \textit{gain})/\textit{loss} \\ [0.5ex]
		\hline
		0.05 & 0.5 & 0.079 & 0.04948720 & 0.1642 & 0.3014 \\
		& & 0.105 & 0.07126956 & 0.2210 & 0.3225 \\ 
		& 0.6 & 0.113 & 0.02959263 & 0.2387 & 0.1239 \\
		& & 0.151 & 0.05424241 & 0.3248 & 0.1670 \\
		& 0.7 & 0.171 & 0.02867069 & 0.3712 & 0.0772 \\
		& & 0.228 & 0.05785821 & 0.5079 & 0.1139 \\
		& 0.8 & 0.226 & 0.03730076 & 0.5031 & 0.0741 \\ 
		& & 0.301 & 0.08313421 & 0.6926 & 0.1200 \\ [0.5ex] 
		\hline
		0.1 & 0.5 & 0.041 & 0.02654781 & 0.0837 & 0.3172 \\ 
		& &  0.055 & 0.05058620 & 0.1130 & 0.4477 \\
		& 0.6 & 0.066 & 0.02488160 & 0.1364 & 0.1824 \\
		& & 0.089 & 0.05302328 & 0.1859 & 0.2852 \\
		& 0.7 & 0.095 & 0.02606301 & 0.1990 & 0.1309 \\
		& & 0.127 & 0.05599555 & 0.2701 & 0.2073 \\
		& 0.8 & 0.117 & 0.03431617 & 0.2477 & 0.1385 \\ 
		& & 0.156 & 0.05958476 & 0.3363 & 0.1772 \\ [0.5ex] 
		\hline
	\end{tabular}
	\label{GainLossResults}
\end{table}

For given nominal coverage $1-\alpha$, size $\widetilde{\alpha}$ of the preliminary test, $|\rho|$, $u$ $m_1$ and $m_2$, it takes roughly 45 minutes to compute the lower bound on $\inf_{s \in {\cal D}} e(0;s)$, subject to the constraints (C1) and (C2), using a computer with i7 processor (3.4 GHz) and 32 GB of RAM. In other words, the time taken to complete Step 1 of the procedure described in Section S3 of the Supplementary Material is roughly 45 minutes. As a consequence, the time taken to complete the entire algorithm described in  Section S3 of the Supplementary Material for given nominal coverage $1-\alpha$, size $\widetilde{\alpha}$ of the preliminary test, $|\rho|$ and $u$ is about 4 hours. The subsequent computation of $u^{**}$ for the given values of nominal coverage $1-\alpha$, size $\widetilde{\alpha}$ of the preliminary test and $|\rho|$, takes an additional minute or so.

\section{\textbf{Remarks}}

\noindent \textsl{\textbf{Remark 7.1}}

\smallskip

\noindent It is reasonable to ask whether or not we have, in posing the question described in the introduction, put forward requirements that are excessively restrictive in the sense that these requirements cannot be achieved by any confidence interval whatsoever. 
In other words, is it possible that, for the testbed scenario of two nested linear regression models and known error variance $\sigma^2$,
there does not exist any confidence interval (whose centre is not necessarily the bootstrap smoothed estimator) for which the answer to the question posed in the introduction is `yes'?

We know that the requirements that we have put forward are not excessively restrictive because, as shown by \cite{KabailaGiri2009a}, \cite{KabailaGiri2013}
and \cite{MainzerKabaila2019}, it is possible to compute formulas for the centre and width of the confidence interval so that
this interval has the attractive features (A1) and (A2), the desired minimum coverage probability $1 - \alpha$ and 
scaled expected length that (a) has a maximum value that is not too much larger than 1 and (b) is substantially less than 1 when the simpler model is correct.
Indeed, the \texttt{R} package \texttt{ciuupi}, described by \cite{MainzerKabaila2019}, computes confidence intervals that have the attractive features (A1) and (A2), the desired minimum coverage probability $1 - \alpha$ and for which the \textit{gain} is set equal to the \textit{loss}, where \textit{gain} and \textit{loss} are as defined in Section 6. 

\medskip

\noindent \textsl{\textbf{Remark 7.2}}

\smallskip

\noindent 
In the present paper we have considered the testbed scenario of two nested linear regression models for known error variance $\sigma^2$. A natural question to ask is the following. Do we obtain similar results when, instead, $\sigma^2$ is unknown, so that it must be estimated from the data?

It is highly plausible that, for two nested linear regression models with $p$ fixed, the following is true. 
The results obtained for the case that the error variance $\sigma^2$ is known provide an excellent approximation to the corresponding results obtained for the case that $\sigma^2$ is unknown and the residual degrees of freedom $n - p$ is moderately large. Results that support this claim are provided in Appendix B of \cite{KabailaWijethunga2019a} and in Section 3 of \cite{KabailaWijethunga2019b}, who consider the case that $\sigma^2$ is unknown.
Therefore the results of the present paper suggest that for two nested linear regression models and $\sigma^2$ unknown the answer to the question posed in the introduction will be `no', provided that $n - p$ is sufficiently large.
However, by construction, the results of  \cite{KabailaWijethunga2019b} suggest that
the answer to this question will be `yes', when $n - p$ is small.

\section{\textbf{Conclusion}}

In Section 6 we consider a confidence interval centred on the bootstrap smoothed estimator, for preliminary test size either 0.05 or 0.1. This confidence interval is required to have (C1) coverage probability that never falls below 0.95 and (C2) scaled expected length that never exceeds
$1 + u$, for given $u > 0$. 
Table 1 gives values of $u^{**}$ such that for all $u$ satisfying $0 < u \le u^{**}$, the scaled expected length of this confidence interval must exceed 1 when the simpler model is correct. 
In other words, this table specifies values of $u > 0$ such that it is impossible to find a formula for the width of this confidence interval such that its scaled expected length is less than 1 when the simpler model is correct. 

In Section 6 we also define a \textit{gain} and \textit{loss} for this confidence interval. Table 1 may be viewed as giving values of the \textit{loss} such that it is impossible for this confidence interval to have any \textit{gain}.
As Table 2 shows, even for the listed values of $u > u^{**}$,
the \textit{gain} cannot be more than a small fraction of the \textit{loss} when $|\rho| \ge 0.6$.

\bigskip

\noindent \textbf{\large Acknowledgements}

\medskip

\noindent This work was supported by an Australian Government Research Training Program Scholarship.


\appendix

\section*{Appendix}


We will express all quantities of interest in terms of the random vector $(\widehat{\theta}, \widehat{\gamma})$, which has the following bivariate normal distribution:
\begin{align}
\label{bivariate_theta_gamma}
\left[ {\begin{array}{c}
	\widehat{\theta} \\
	\widehat\gamma
	\end{array} } \right] \sim N \left(
\left[ {\begin{array}{c}
	\theta \\
	\gamma
	\end{array} } \right],
\left[ {\begin{array}{cc}
	\sigma^2 \, v_\theta & \rho \, \sigma\, {v_\theta}^{1/2} \\
	\rho \, \sigma \, {v_\theta}^{1/2} & 1
	\end{array} } \right]
\right).
\end{align}

\subsection*{A.1 Proof of Theorem 1}

The following proof is based, in part, on the derivations described in Section 4.3 of \cite{Giri2008}.
The coverage probability of the confidence interval $\text{CI}(s)$ is
\begin{align*}
\notag
P(\theta \in \text{CI}(s)) 
&= P\left( \widehat{\theta} - \sigma \, v_\theta^{1/2}\, b(\widehat{\gamma}) - \sigma \, v_\theta^{1/2}\, s(\widehat{\gamma}) \leq \theta \leq \widehat{\theta} - \sigma \, v_\theta^{1/2}\, b(\widehat{\gamma}) + \sigma \, v_\theta^{1/2}\, s(\widehat{\gamma}) \right)\\
\notag
&= P\left( -\widehat{\theta} + \sigma \, v_\theta^{1/2}\, b(\widehat{\gamma}) + \sigma \, v_\theta^{1/2}\, s(\widehat{\gamma}) \geq -\theta \geq -\widehat{\theta} + \sigma \, v_\theta^{1/2}\, b(\widehat{\gamma}) - \sigma \, v_\theta^{1/2}\, s(\widehat{\gamma}) \right)\\
\notag
&= P\left( \sigma \, v_\theta^{1/2} \big(b(\widehat{\gamma}) - s(\widehat{\gamma})\big) \leq \widehat{\theta} - \theta \leq \sigma \, v_\theta^{1/2} \big(b(\widehat{\gamma}) + s(\widehat{\gamma})\big) \right)\\
\notag
&= P\left(  b(\widehat{\gamma}) - s(\widehat{\gamma}) \leq \frac{\widehat{\theta} - \theta}{\sigma \, v_\theta^{1/2}} \leq  b(\widehat{\gamma}) + s(\widehat{\gamma}) \right)\\
&= P \big(  b(\widehat{\gamma}) - s(\widehat{\gamma}) \leq G \leq  b(\widehat{\gamma}) + s(\widehat{\gamma}) \big),
\end{align*}
where $G = (\widehat{\theta} - \theta)/\big(\sigma \,{v_\theta^{1/2}} \big)$.
It follows from \eqref{bivariate_theta_gamma} that
\begin{align}
\label{bivariate_G_gamma}
\left[ {\begin{array}{c}
	G \\
	\widehat\gamma
	\end{array} } \right] \sim N \left(
\left[ {\begin{array}{c}
	0 \\
	\gamma
	\end{array} } \right],
\left[ {\begin{array}{cc}
	1 & \rho  \\
	\rho  & 1
	\end{array} } \right]
\right).
\end{align}
For given $\rho$ and function $s$, the coverage probability of $\text{CI}(s)$ is a function of 
$\gamma$. We denote this coverage probability by $c(\gamma; s, \rho)$.

Since $b$ and $s$ are odd and even functions, respectively,
\begin{align*}
c(\gamma; s, \rho)
&= P \big(  -b(-\widehat{\gamma}) - s(-\widehat{\gamma}) \leq G \leq  -b(-\widehat{\gamma}) + s(-\widehat{\gamma}) \big)
\\
&= P \big(  b(-\widehat{\gamma}) - s(-\widehat{\gamma}) \leq -G \leq  b(-\widehat{\gamma}) + s(-\widehat{\gamma}) \big)
\\
&= P \big(  b(\widehat{\gamma}^{\prime}) - s(\widehat{\gamma}^{\prime}) \leq G^{\prime} \leq  b(\widehat{\gamma}^{\prime}) + s(\widehat{\gamma}^{\prime}) \big),
\end{align*}
where $G^{\prime} = -G$ and $\widehat{\gamma}^{\prime} = - \widehat{\gamma}$. It follows from \eqref{bivariate_G_gamma} that 
\begin{align*}
\left[ {\begin{array}{c}
	G^{\prime} \\
	\widehat{\gamma}^{\prime}
	\end{array} } \right] \sim N \left(
\left[ {\begin{array}{c}
	0 \\
	-\gamma
	\end{array} } \right],
\left[ {\begin{array}{cc}
	1 & \rho  \\
	\rho  & 1
	\end{array} } \right]
\right).
\end{align*}
Hence $c(\gamma; s, \rho) = c(-\gamma; s, \rho)$.

Since $b(x) = \rho \, k(x)$,
\begin{align*}
c(\gamma; s, \rho)
&= P \big(  \rho \, k(\widehat{\gamma}) - s(\widehat{\gamma}) \leq G \leq  \rho \, k(\widehat{\gamma}) + s(\widehat{\gamma}) \big)
\\
&=  P \big( - \rho \, k(\widehat{\gamma}) - s(\widehat{\gamma}) 
\leq - G 
\leq  - \rho \, k(\widehat{\gamma}) + s(\widehat{\gamma}) \big)
\\
&= P \big( (- \rho) \, k(\widehat{\gamma}) - s(\widehat{\gamma}) 
\leq G^{\prime} 
\leq  (- \rho) \, k(\widehat{\gamma}) + s(\widehat{\gamma}) \big),
\end{align*}
where $G^{\prime} = -G$. It follows from \eqref{bivariate_G_gamma} that 
\begin{align*}
\left[ {\begin{array}{c}
	G^{\prime} \\
	\widehat{\gamma}
	\end{array} } \right] \sim N \left(
\left[ {\begin{array}{c}
	0 \\
	\gamma
	\end{array} } \right],
\left[ {\begin{array}{cc}
	1 & - \rho  \\
	- \rho  & 1
	\end{array} } \right]
\right).
\end{align*}
Hence $c(\gamma; s, \rho) = c(\gamma; s, -\rho)$.

 It follows from \eqref{bivariate_G_gamma} that  the probability distribution of $G$, conditional on $\widehat{\gamma}=h$, is $N\big(\rho(h-\gamma), 1-\rho^2 \big)$. Note that
\begin{align*}
&P \big(  b(\widehat{\gamma}) - s(\widehat{\gamma}) \leq G \leq  b(\widehat{\gamma}) + s(\widehat{\gamma}) \big) 
\\
&= \int_{-\infty}^{\infty} P \big(  b(h) - s(h) \leq G \leq  b(h) + s(h) \big| \widehat{\gamma}=h \big) \, \phi(h-\gamma)\, dh
\\
&= \int_{-\infty}^{\infty} P \big(  b(h) - s(h) \leq \widetilde{G} \leq  b(h) + s(h) \big) \, \phi(h-\gamma)\, dh,
\end{align*}
where $\widetilde{G} \sim N\big(\rho(h-\gamma), 1-\rho^2 \big)$. Thus
\begin{align}
\notag
&c(\gamma; s, \rho)
\\
\notag
&= \int_{-\infty}^{\infty} \ell(h, \gamma; s(h)) \, \phi(h-\gamma)\, dh
\\
\notag
 &= \int_{-\infty}^{-c} \ell(h, \gamma; s(h)) \, \phi(h-\gamma)\, dh + \int_{-c}^{c} \ell(h, \gamma; s(h)) \, \phi(h-\gamma)\, dh + \int_{c}^{\infty} \ell(h, \gamma; s(h)) \, \phi(h-\gamma)\, dh \\
\label{CP_FirstExpression}
&= \int_{-\infty}^{-c} \ell^{\dag}(h, \gamma) \, \phi(h-\gamma)\, dh + \int_{-c}^{c} \ell(h, \gamma; s(h)) \, \phi(h-\gamma)\, dh + \int_{c}^{\infty} \ell^{\dag}(h, \gamma) \, \phi(h-\gamma)\, dh. 
\end{align}
The usual $1-\alpha$ confidence interval based on the full model ${\cal M}_2$ has coverage probability $1-\alpha$. Thus
\begin{align*}
1 - \alpha 
&=  P \big( -z(\alpha) \leq G \leq z(\alpha) \big) 
\\
&= \int_{-\infty}^{\infty} P\big( -z(\alpha) \leq G \leq z(\alpha) \, \big| \, \widehat{\gamma}=h \big) \, \phi(h-\gamma) \, dh
\\
&= \int_{-\infty}^{\infty} \ell^\dag(h, \gamma)\, \phi(h-\gamma) \, dh.
\end{align*}
Therefore
\begin{equation*}
1-\alpha = \int_{-\infty}^{-c} \ell^{\dag}(h, \gamma) \, \phi(h-\gamma)\, dh + \int_{-c}^{c} \ell^{\dag}(h, \gamma) \, \phi(h-\gamma)\, dh + \int_{c}^{\infty} \ell^{\dag}(h, \gamma) \, \phi(h-\gamma)\, dh. 
\end{equation*}
It follows from this equality and \eqref{CP_FirstExpression} that
\begin{align*}
c(\gamma; s, \rho)
&= 1 - \alpha + \int_{-c}^{c} \big( \ell(h, \gamma; s(h)) - \ell^{\dag}(h, \gamma) \big) \, \phi(h-\gamma)\, dh
\\
& = 1 - \alpha - 
\Big( \int_{0}^{c} \left( \ell(h, \gamma; s(h)) - \ell^{\dag}(h, \gamma) \right) \, \phi(h-\gamma)\, dh 
\\
&\qquad  \qquad \qquad + \int_{-c}^{0} \left( \ell(h, \gamma; s(h)) - \ell^{\dag}(h, \gamma) \right) \, \phi(h-\gamma)\, dh \Big).
\end{align*}
Change the variable of integration to $y=-h$ in the second integral. The result 
$c(\gamma; s, \rho) = 1 - \alpha - R_1(s, \gamma)$
now follows from the fact that both $s$ and $\phi$ are even functions.

\hfill $\qed$

\subsection*{A.2 Proof of Theorem 2}

The following proof is based, in part, on the derivations described in Section 4.3 of \cite{Giri2008}.
Note that 
\begin{align}
\notag
e(\gamma; s) &=  \frac{1}{z(\alpha)} \int_{-\infty}^{\infty} s(h) \, \phi(h-\gamma) \, dh
\\
\label{SEL_First}
&= \int_{-\infty}^{-c} \phi(h-\gamma)\, dh + \frac{1}{z(\alpha)}
\int_{-c}^{c} s(h) \, \phi(h-\gamma)\, dh + \int_{c}^{\infty}  \phi(h-\gamma) \, dh,
\end{align}
since $s(x)=z(\alpha)$ for all $|x| \geq c$.
Obviously,
\begin{equation*}
1 = \int_{-\infty}^{-c}  \phi(h-\gamma)\, dh + 
\int_{-c}^{c}  \phi(h-\gamma)\, dh + \int_{c}^{\infty} \phi(h-\gamma)\, dh. 
\end{equation*}
It follows from this equality and \eqref{SEL_First} that
\begin{align*}
e(\gamma; s) &= 1 + \int_{-c}^{c} \left( \frac{s(h)}{z(\alpha)} - 1 \right) \phi(h-\gamma)\, dh
\\
&= 1 + \int_{-c}^{0} \left( \frac{s(h)}{z(\alpha)} - 1 \right) \phi(h-\gamma)\, dh + \int_{0}^{c} \left( \frac{s(h)}{z(\alpha)} - 1 \right) \phi(h-\gamma)\, dh.
\end{align*}
Change the variable of integration to $y=-h$ in the first integral on the right-hand side. The 
fact that both $s$ and $\phi$ are even functions implies that
\eqref{ExactFormulaRsGamma} is true.

\hfill $\qed$

\subsection*{A.3 Properties of $dq(x; h, \bm{\gamma}, \boldsymbol{\nu})/dx$
considered as a function of $x$}

It is straightforward to show that
\begin{equation*}
\ell(h, \gamma; x) = \Phi \left( \frac{b(h) + x - \rho(h-\gamma)}{(1-\rho^2)^{1/2}} \right) - \Phi \left( \frac{b(h) - x - \rho(h-\gamma)}{(1-\rho^2)^{1/2}} \right)
\end{equation*}
and 
\begin{equation*}
\ell^{\dag}(h, \gamma) = \Phi \left( \frac{z(\alpha) - \rho(h-\gamma)}{(1-\rho^2)^{1/2}} \right) - \Phi \left( \frac{- z(\alpha) - \rho(h-\gamma)}{(1-\rho^2)^{1/2}} \right).
\end{equation*}
It follows that
\begin{align*}
\frac{d\, q(x; h, \bm{\gamma}, \boldsymbol{\nu})}{dx}
= t_1(h,  \bm{\gamma}, \bm{\nu}) - t_2(x; h, \bm{\gamma}, \boldsymbol{\nu}),
\end{align*}
where $t_1(h,  \bm{\gamma}, \bm{\nu})$ is defined to be
\begin{align*}
\frac{1}{z(\alpha)} \Big( 2\phi(h) 
+ \sum_{j=1}^{m_2} \nu_2(j)\, \big( \phi(h-\gamma_2(j)) + \phi(h+\gamma_2(j)) \big) \Big)
\end{align*}
and $t_2(x; h, \bm{\gamma}, \boldsymbol{\nu})$ is defined to be
\begin{align*}
\sum_{j=1}^{m_1} \nu_1(j) \left( \phi(h-\gamma_1(j)) \frac{d\, \ell(h, \gamma_1(j); x)}{dx} + \phi(h+\gamma_1(j)) \frac{d\, \ell(-h, \gamma_1(j); x)}{dx} \right),
\end{align*}
with
\begin{equation*}
\frac{d\, \ell(h, \gamma; x)}{dx} = \frac{1}{(1-\rho^2)^{1/2}}\Bigg( \phi \left( \frac{b(h) + x - \rho(h-\gamma)}{(1-\rho^2)^{1/2}} \right) + \phi \left( \frac{b(h) - x - \rho(h-\gamma)}{(1-\rho^2)^{1/2}} \right) \Bigg).
\end{equation*}

Suppose that $h \in [0,c]$ and $(\bm{\gamma}, \boldsymbol{\nu})$ are given. Then 
$t_1(h,  \bm{\gamma}, \bm{\nu})$
is a fixed positive number. Observe that
$t_2(x; h, \bm{\gamma}, \boldsymbol{\nu})$
is a function of $x \in [0, \infty)$ that can only take positive values and
$d\ell(h, \gamma; x)/dx$
approaches 0 as $x \rightarrow \infty$.
We will use the following theorem to find $\widetilde{x} < \infty$, such that $dq(x; h, \bm{\gamma}, \bm{\nu})/dx > 0$ for all $x \ge \widetilde{x}$. This implies that a value
of $x$ that 
minimizes $q(x; h, \bm{\gamma}, \bm{\nu})$
cannot belong to the interval $[\widetilde{x}, \infty)$.

\begin{theorem}
	Let $\mu(h, \rho, \gamma) = b(h) - \rho(h-\gamma)$. Then $t_2(x; h, \bm{\gamma}, \boldsymbol{\nu})$
	is a decreasing function of $x \in \big[x^*, \infty \big)$, where 
	\begin{equation*}
	x^* = \max \big( |\mu(h, \rho, \gamma_1(1))|, \, |\mu(-h, \rho, \gamma_1(1))|, \dots,   |\mu(h, \rho, \gamma_1(m_1))|, \, |\mu(-h, \rho, \gamma_1(m_1))|\big).
	\end{equation*}
\end{theorem}

\begin{proof}
	We first prove that, for every $h \in \mathbb{R}$, $d \ell(h, \gamma; x)/dx$
	is a decreasing function of $x \in \big[|\mu(h, \rho, \gamma)| , \infty)$.
	Observe that,
	for all $x \ge 0$,
	\begin{align*}
	&\phi \left( \frac{b(h) + x - \rho(h-\gamma)}{(1-\rho^2)^{1/2}} \right) + \phi \left( \frac{b(h) - x - \rho(h-\gamma)}{(1-\rho^2)^{1/2}} \right)
	\\
	&= \phi \left( \frac{b(h) - \rho(h-\gamma) + x }{(1-\rho^2)^{1/2}} \right) + \phi \left( \frac{- \big(b(h) - \rho(h-\gamma) \big) + x}{(1-\rho^2)^{1/2}} \right),
\  \text{since $\phi$ is an even function,}
\\
	&= \phi \left( \frac{\mu(h, \rho, \gamma) + x }{(1-\rho^2)^{1/2}} \right) + \phi \left( \frac{- \mu(h, \rho, \gamma) + x}{(1-\rho^2)^{1/2}} \right) 
	\\
	&= \phi \left( \frac{|\mu(h, \rho, \gamma)| + x }{(1-\rho^2)^{1/2}} \right) + \phi \left( \frac{- |\mu(h, \rho, \gamma)| + x}{(1-\rho^2)^{1/2}} \right).
	\end{align*}
	 This is a decreasing function of $x \in \big[|\mu(h, \rho, \gamma)| , \infty)$.
	 Consequently, $d \ell(h, \gamma; x)/dx$ and
	  $d \ell(-h, \gamma; x)/dx$
	 are decreasing functions of $x \in \big[|\mu(h, \rho, \gamma)| , \infty)$ and
	 $x \in \big[|\mu(-h, \rho, \gamma)| , \infty)$, respectively.

	 Thus $d \ell(h, \gamma_1(j); x)/dx$ is a decreasing function of $x \in \big[|\mu(h, \rho, \gamma_1(j))| , \infty \big)$, for $j = 1, \dots, m_1$.
	Therefore
	\begin{equation*}
	\sum_{j=1}^{m_1} \nu_1(j)\, \phi(h-\gamma_1(j)) \frac{d\, \ell(h, \gamma_1(j); x)}{dx}
	\end{equation*}
	is a decreasing function of $x \in \Big[\underset{j = 1, \dots, m_1}{\mathrm{max}} |\mu(h, \rho, \gamma_1(j)) | , \infty \Big)$.
 Similarly, $d \ell(-h, \gamma_1(j); s)/ds$ is a decreasing function of $s \in \big[|\mu(-h, \rho, \gamma_1(j))| , \infty \big)$, for $j = 1, \dots, m_1$.
	Therefore
	\begin{equation*}
	\sum_{j=1}^{m_1} \nu_1(j)\, \phi(h+\gamma_1(j)) \frac{d\, \ell(-h, \gamma_1(j); s)}{ds}
	\end{equation*}
	is a decreasing function of $s \in \Big[\underset{j = 1, \dots, m_1}{\mathrm{max}} |\mu(-h, \rho, \gamma_1(j)) | , \infty \Big)$. Therefore
	$t_2(x; h, \bm{\gamma}, \boldsymbol{\nu})$
	is a decreasing function of $x \in \big[x^*, \infty \big)$.
\end{proof}

We use this theorem to find $\widetilde{x} < \infty$, such that $dq(x; h, \bm{\gamma}, \bm{\nu})/dx > 0$ for all $x \ge \widetilde{x}$ as follows. 
First evaluate $x^*$ and then $dq(x^*; h, \bm{\gamma}, \boldsymbol{\nu})/dx$. 
 If $dq(x^*; h, \bm{\gamma}, \boldsymbol{\nu})/dx > 0$ then set $\widetilde{x} = x^*$ and stop;
otherwise use the \texttt{R} function \texttt{uniroot} to find the solution for 
$x \in [x^*, \infty)$ of $dq(x; h, \bm{\gamma}, \boldsymbol{\nu})/dx = 0$ and then set $\widetilde{x}$ equal to this solution.

\end{document}